\documentclass[12pt, reqno]{amsart}

\usepackage{amsmath}
\usepackage{amsfonts}
\usepackage{amsthm}
\usepackage{enumerate}
\usepackage{cleveref}
\usepackage{amssymb}
\usepackage{setspace}
\usepackage{comment}
\usepackage{fancyhdr}
\usepackage{url}
\usepackage[english]{babel}
\usepackage[english=american]{csquotes}
\usepackage[margin=1in]{geometry}

 \newtheorem{thm}{Theorem}[section]
 \newtheorem{cor}[thm]{Corollary}
 \newtheorem{lem}[thm]{Lemma}
 \newtheorem{prop}[thm]{Proposition}
 \theoremstyle{definition}
 
 \theoremstyle{remark}
 \newtheorem{rem}[thm]{Remark}
 
 \numberwithin{equation}{section}

\begin{document}



\title{Toeplitz Operators on the Fock Space in a Symmetrically-Normed Ideal}
\maketitle
Adam Orenstein \footnote{Adam Orenstein \\214 Fitzelle Hall\\SUNY Oneonta\\Oneonta, NY 13820\\Adam.Orenstein@oneonta.edu\\(607)436-3520} \smallskip

\textbf{Abstract}:  We look at Toeplitz operators $T_\nu$ on the Fock Space (also known as the Segal-Bargmann space) which have a positive Borel measure $\nu$ as a symbol.  We characterize when $\left(T_\nu\right)^s$ for $0<s\leq 1$ is in the symmetrically normed ideal associated with any given symmetric norming function.

\textbf{Keywords:} Toeplitz Operator; Fock Space; Symmetrically-Normed Ideal; Symmetric Norming Function. \smallskip

\textbf{Math subject Classification (2010)}: 47B35; 47B10.


\section{Introduction and Preliminaries}
\subsection{Fock Space and Toeplitz Operators}\label{snnotat}

Let $|\cdot|$ be the usual Euclidean norm and $d\mu$ be the measure on $\mathbb{C}^n$ defined by \[d\mu(z)=\frac{e^{-|z|^2}}{\pi^n}dV(z)\] where $dV$ is the standard volume measure on $\mathbb{C}^n$.  Let $\langle \cdot, \cdot\rangle$ be the inner product on $L^2(\mathbb{C}^n,d\mu)$ defined by \[\langle f,g\rangle=\int_{\mathbb{C}^n}f(z)\overline{g(z)}d\mu(z) \text{ for all }f,g\in L^2(\mathbb{C}^n,d\mu)\] and $\|\cdot\|_2$ be the corresponding induced norm.  Then the Fock space (also known as the Segal-Bargmann space) $H^2=H^2(\mathbb{C}^n,d\mu)$ is defined as the set of all functions in $L^2(\mathbb{C}^n,d\mu)$ that are analytic on $\mathbb{C}^n$.  It is well-known in the literature that $H^2$ is a closed subspace of $L^2(\mathbb{C}^n,d\mu)$ with respect to $\|\cdot\|_2$.  Thus $H^2$ is itself a Hilbert Space with respect to $\langle\cdot,\cdot \rangle$.

Let $\mathfrak{L}(H^2)$ be the set of all bounded operators on $H^2$ with operator norm denoted by $\|\cdot\|$ and let $P:L^2(\mathbb{C}^n,d\mu)\rightarrow H^2$ be the orthogonal projection of $L^2(\mathbb{C}^n,d\mu)$ onto $H^2$.  Then for any $f\in L^\infty(\mathbb{C}^n,dV)$, the Toeplitz operator on $H^2$ with symbol $f$ is the operator $T_f$ defined by \[T_f(g)=P(fg) \text{ for all }g\in H^2.\]  A straightforward calculation shows $\|T_f\|\leq \|f\|_{\infty}$.  Thus $T_f\in\mathfrak{L}(H^2)$.  For any $A\in\mathfrak{L}(H^2)$, let $A^*$ denote the adjoint of $A$.  Then by direct calculation, \[\left(T_f\right)^*=T_{\overline{f}} \text{ and }  f\geq 0 \text{ implies } T_f\geq 0.\]

More generally if $\nu$ is a complex Borel measure on $\mathbb{C}^n$ such that \begin{equation}\label{wellDef}\int_{\mathbb{C}^n}\left|K(z,w)\right| e^{-|w|^2}d|\nu|(w)<\infty \text{ for all } z\in\mathbb{C}^n\end{equation} then we define the Toeplitz operator $T_\nu$ on $H^2$ by \[\left(T_\nu(g)\right)(z)=\int_{\mathbb{C}^n}K(z,w)g(w)e^{-|w|^2}d\nu(w) \text{ for all } z\in\mathbb{C}^n \text{ and for all } g\in H^2 \] where $K(z,w)=e^{z_1\overline{w_1}+\cdots+z_n\overline{w_n}}$ is the reproducing kernel of $H^2$ \cite{Josh}.  Because of \eqref{wellDef} and the Cauchy-Schwarz inequality $T_\nu$ is well-defined on a dense subset of $H^2$, namely the set of all functions $g\in H^2$ such that \[g(w)=\sum_{j=1}^N c_jK(w,z_j) \text{ for all }w\in\mathbb{C}^n\] where $\{z_j\}_{j=1}^N\subseteq \mathbb{C}^n$, $\{c_j\}_{j=1}^N\subseteq \mathbb{C}$ and $N\in\mathbb{N}$ are arbitrary \cite[page 41]{zhuFock}.  We will assume all of our measures satisfy \eqref{wellDef}.

For any $z=(a_1+ib_1,a_2+ib_2,\ldots,a_n+ib_n)\in \mathbb{C}^{n}$ where $\{a_j\}_{j=1}^n, \{b_j\}_{j=1}^n\subseteq\mathbb{R}$, define $|z|_\infty$ by \[|z|_\infty=\max\{|a_j|, |b_j|\}_{j=1}^n.\]  Then \begin{equation}\label{normIn}|z| \leq \sqrt{2n}|z|_\infty \text{ and } |z|_\infty\leq |z| \text{ for all }z\in\mathbb{C}^n. \end{equation}  For any $\delta>0$ and $z\in\mathbb{C}^n$, let $B(z,\delta)=\{w\in\mathbb{C}^n:|w-z|_\infty<\delta\}$ and $\overline{B(z,\delta)}=\{w\in\mathbb{C}^n:|w-z|_\infty\leq\delta\}$.  For any $\alpha>0$, let $\widetilde{\nu}_\alpha$ be the function on $\mathbb{C}^n$ defined by \[\widetilde{\nu}_\alpha(z)=\int_{\mathbb{C}^n}e^{-\alpha|z-w|^2}d\nu(w) \text{ for all } z\in\mathbb{C}^n.\]  For any $r>0$ we define the function $\widehat{\nu}_r:\mathbb{C}^n\rightarrow \mathbb{C}$ by \[\widehat{\nu}_r(z)=\nu(B(z,r)) \text{ for all } z\in\mathbb{C}^n.\]

For any $r>0$, the $r$-lattice of $\mathbb{C}^n$ is defined as $\{r(b_1+ic_1,\ldots,b_n+ic_n):\{b_j\}_{j=1}^n,\{c_j\}_{j=1}^n\subseteq\mathbb{Z}\}$.  Since the $r$-lattice of $\mathbb{C}^n$ is countable, we can write the $r$-lattice of $\mathbb{C}^n$ as a sequence $\{a_j\}_{j=1}^\infty$.  We also point out the following two useful facts about $\{a_j\}_{j=1}^\infty$, the first of which is easy to prove and the second is due to \cite[Theorem 11.1]{MunMan}: \begin{equation}\label{lattice} \mathbb{C}^n=\bigcup_{j=1}^\infty \overline{B\left(a_j,\frac{r}{2}\right)}=\bigcup_{j=1}^\infty B\left(a_j,r\right) \text{ and }  \mu\left(\overline{B\left(a_j,\frac{r}{2}\right)}\cap \overline{B\left(a_k,\frac{r}{2}\right)}\right)=0 \hspace{.1cm} \text{ if } k\neq j.\end{equation}

We will now state some important properties about Toeplitz operators and the Fock space that we will need.  \begin{lem}\label{tilde} Suppose $\nu$ is a positive Borel measure on $\mathbb{C}^n$, $r>0$ and $\alpha>0$.  Then \[ \widehat{\nu}_r(z)\leq e^{\alpha2nr^2}\widetilde{\nu}_\alpha(z) \text{ for all } z\in\mathbb{C}^n.\]\end{lem}  A proof of Lemma \ref{tilde} can be easily made using \eqref{normIn} and the proof of \cite[Lemma 2.2]{Josh}.

For any $z,w\in\mathbb{C}^n$, let $k_z(w)=\frac{K(w,z)}{\sqrt{K(z,z)}}=K(w,z)e^{\frac{-|z|^2}{2}}$.  The next result we will need is Theorem \ref{atomCor}, which is \cite[Theorem 8.4]{Hank}.

\begin{thm}\label{atomCor}
There exists $\delta>0$ such that if $0<\rho<\delta$ then for some positive constants $C_1(\rho)$ and $C_2(\rho)$, \begin{equation}\label{separ} C_1(\rho)\|g\|_2^2\leq \sum_{j=1}^\infty |\langle g,k_{b_j}\rangle|^2\leq C_2(\rho)\|g\|_2^2\end{equation} for all $g\in H^2$ where $\{b_j\}_{j=1}^\infty$ is the $\rho$-lattice of $\mathbb{C}^n$.
\end{thm}

\subsection{Symmetric norming functions, Symmetrically-normed ideals and s-numbers}\label{snSection}

We will now define symmetric norming functions and symmetrically-normed ideals.  For this section we let $\mathcal{H}$ denote any separable complex Hilbert space, $\mathfrak{L}(\mathcal{H})$ denote the set of all bounded linear operators on $\mathcal{H}$ with operator norm denoted by $\|\cdot\|_\mathcal{H}$, $\mathfrak{LC}(\mathcal{H})$ be the set of all compact operators on $\mathcal{H}$ and, given any $A\in\mathfrak{L}(\mathcal{H})$, let $\sigma_{\mathcal{H}}(A)$ denote the spectrum of $A$ as an element of $\mathfrak{L}(\mathcal{H})$.  Since we will be working with sequences, we will from now on say a sequence of real numbers $\{c_j\}_{j=1}^\infty$ is nonincreasing if $c_{j+1}\leq c_j$ for all $j\geq1$ and nondecreasing if $c_j\leq c_{j+1}$ for all $j\geq1$.

\subsubsection{s-numbers}

We will define now the s-numbers of a bounded linear operator on $\mathcal{H}$ and give some important results about the s-numbers.  We will first define the numbers $\lambda_j(H)$ $(j=1,2,\ldots )$ for any $H\geq0$.  A point $\lambda\in\sigma_\mathcal{H}(A)$, where $A\in\mathfrak{L}(\mathcal{H})$ is self-adjoint, is called a point of the condensed spectrum of $A$ if $\lambda$ is either an accumulation point of $\sigma_\mathcal{H}(A)$ or an eigenvalue of $A$ of infinite multiplicity.

Let $H\geq0$ and let $\eta=\sup\sigma_\mathcal{H}(H)$.  If $\eta$ is in the condensed spectrum of $H$, then we put \[\lambda_j(H)=\eta \hspace{1cm} (j=1,2,\ldots).\]  If $\eta$ does not belong to the condensed spectrum of $H$, then it is an eigenvalue of finite multiplicity. In this case we put \[\lambda_j(H)=\eta \hspace{1cm} (j=1,2,\ldots,p)\] where $p$ is the multiplicity of the eigenvalue $\eta$.

In the latter case the remaining numbers $\lambda_j(H)$ $(j=p+1,\ldots)$ are defined by \[\lambda_{p+j}(H)=\lambda_j(H_1) \hspace{1cm} (j=1,2,\ldots)\] where the operator $H_1$ is given by \begin{equation}\label{snDef}H_1=H-\eta P,\end{equation} and $P$ is the orthoprojector onto the eigenspace of the operator $H$ corresponding to the eigenvalue $\eta$.  For any $A\in\mathfrak{L}(\mathcal{H})$, we define the s-numbers of $A$, $s_j(A) \ (j=1,2,\ldots)$, by \[s_j(A)=\lambda_j(H) \ (j=1,2,\ldots )\] where $H=(A^*A)^{\frac{1}{2}}$.  The definition given above can also be found in \cite[page 59]{Gohb}.  We will denote the sequence of s-numbers of $A$ as $\{s_j(A)\}_{j=1}^\infty$ enumerated so that $\{s_j(A)\}_{j=1}^\infty$ is nonincreasing and so as to include multiplicities.  

We will now give some important results about $s$-numbers.  The next lemma is based on the definition of s-numbers.

\begin{lem}\label{eigen}  Let $H,C,D\in \mathfrak{L}(\mathcal{H})$.  Then
\begin{enumerate}

\item $s_1(D)=\|D\|_\mathcal{H}$.

\item $s_j(C|D|C^*))\leq \|C\|_\mathcal{H}^{2}s_j(|D|)$ for all $j\geq1$.

\item $s_j(D)=s_j(|D|)$ for all $j\geq1$.

\item If $0\leq C\leq D$, then $s_j(C)\leq s_j(D)$ $\text{ for all } j\geq1$.

\end{enumerate}
\end{lem}

The next result is very useful and can easily be proved using Functional Calculus and the material above.  \begin{prop}\label{prop}
Let $A\in\mathfrak{LC}(\mathcal{H})$ such that $A\geq 0$.  Then for any nonnegative function $f$ which is bounded on $\mathbb{C}$, continuous at 0 and satisfies $f(0)=0$, the following statements hold:

\begin{enumerate}
\item $f(A)\geq 0$

\item $f(A)\in\mathfrak{LC}(\mathcal{H})$

\item $\{f(s_j(A)):j\geq1\}=\{s_j(f(A)):j\geq1\}$.

\item If $f$ is nondecreasing on $\sigma_{\mathcal{H}}(A)$, then $f(s_j(A))=s_j(f(A))$ for all $j\geq1$.

\end{enumerate}
\end{prop}

\subsubsection{symmetric norming functions and symmetrically-normed ideals}

Let $\widehat{c}$ be the set of all sequences of real numbers with only a finite number of nonzero terms.  A function $\Phi:\widehat{c}\rightarrow [0,\infty)$ is called a symmetric norming function if $\Phi$ is a norm that has the following properties:

\begin{itemize}
\item $\Phi(1,0,0,\ldots)=1$

\item $\Phi(\{\xi_j\}_{j=1}^\infty)=\Phi\left(\{|\xi_{\theta(j)}|\}_{j=1}^\infty\right)$ for any bijection $\theta:\mathbb{N}\rightarrow\mathbb{N}$ and any $\{\xi_j\}_{j=1}^\infty\in \widehat{c}.$
\end{itemize}

A norm $|\cdot|_\mathcal{C}$ defined on a two-sided ideal $\mathcal{C}$ in $\mathfrak{L}(\mathcal{H})$ is called a symmetric norm if the following two conditions hold:
\begin{itemize}
\item For any $A,B\in\mathfrak{L}(\mathcal{H})$ and $D\in\mathcal{C}$, $|ADB|_\mathcal{C}\leq\|A\|_{\mathcal{H}}|D|_\mathcal{C}\|B\|_{\mathcal{H}}$

\item For any rank one operator $T$, $|T|_\mathcal{C}=\|T\|_{\mathcal{H}}$.
\end{itemize}
If in addition $\mathcal{C}\neq\{0\}$ and $\mathcal{C}$ is a Banach space with respect to $|\cdot|_\mathcal{C}$, then $\mathcal{C}$ is called a symmetrically-normed ideal.  For any symmetric norming function $\Phi$, let \[\mathcal{C}_\Phi=\{A\in \mathfrak{L}(\mathcal{H}):\Phi\left(\{s_j(A)\}_{j=1}^\infty\right)<\infty\}.\]  By \cite[page 80-82]{Gohb}, $\mathcal {C}_\Phi$ is a symmetrically-normed ideal with corresponding symmetric norm $|\cdot|_\Phi$ defined by \[|A|_\Phi=\Phi\left(\{s_j(A)\}_{j=1}^\infty\right).\]

One of the main properties of symmetric norming functions is given in the following lemma.

\begin{lem}\label{symm}
Let $\Phi$ be a symmetric norming function and let $\{\xi_j\}_{j=1}^\infty, \{\eta_j\}_{j=1}^\infty\in \widehat{c}$.  If \[|\xi_j|\leq|\eta_j| \text{ for all } j\geq 1,\] then \[\Phi(\{\xi_j\}_{j=1}^\infty)\leq \Phi(\{\eta_j\}_{j=1}^\infty).\]
\end{lem}

A proof of this property can be found in \cite[page 71-72]{Gohb}.  It follows from Lemma \ref{symm} that if $\xi^m:=\{\xi_1,\ldots,\xi_m,0,0,\ldots\}$ where $\xi=\{\xi_j\}_{j=1}^\infty$ is any sequence of real numbers and $m\geq1$, then $\{\Phi(\xi^m)\}_{m=1}^\infty$ is a nondecreasing sequence.  So we extend the definition of $\Phi$ by defining \begin{equation}\label{extendDef} \Phi\left(\{\xi_j\}_{j=1}^\infty\right)=\lim_{m\rightarrow\infty}\Phi(\xi^m)\end{equation} for any sequence of real numbers $\{\xi_j\}_{j=1}^\infty$.  With \eqref{extendDef}, one can easily generalize Lemma \ref{symm} to the set of all sequences of real numbers.

Let $\left(\widehat{c}\right)^*=\{\{\xi_j\}_{j=1}^\infty\in\widehat{c}: \xi_j\neq0 \text{ for one or more }j\}$.  For any two symmetric norming functions $\Psi$ and $\Phi$, we will as in \cite[page 76]{Gohb} say $\Phi$ and $\Psi$ are equivalent if \begin{equation}\label{equivSN}\sup_{\{\xi_j\}_{j=1}^\infty\in \left(\widehat{c}\right)^*} \frac{\Phi(\{\xi_j\}_{j=1}^\infty)}{\Psi(\{\xi_j\}_{j=1}^\infty)}<\infty \text{ and } \sup_{\{\xi_j\}_{j=1}^\infty\in \left(\widehat{c}\right)^*} \frac{\Psi(\{\xi_j\}_{j=1}^\infty)}{\Phi(\{\xi_j\}_{j=1}^\infty)}<\infty.\end{equation}  We will write $\Phi \sim \Psi$ to mean $\Phi$ and $\Psi$ are equivalent.

Here are two examples of symmetric norming functions:

\begin{itemize}
\item For any $p>0$, let $\Phi_p:\widehat{c}\rightarrow [0,\infty)$ be defined by \[\Phi_p(\{\xi_j\}_{j=1}^\infty)=\left(\sum_j |\xi_j|^p\right)^{\frac{1}{p}}.\] Then if $p\geq1$, $\Phi_p$ is a symmetric norming function called the Schatten $p$-norm.

\item Let $\Phi_\infty:\widehat{c}\rightarrow [0,\infty)$ be defined by \[\Phi_\infty(\{\xi_j\}_{j=1}^\infty)=\sup_{j\geq1}|\xi_j|.\]  Then $\Phi_\infty$ is a symmetric norming function.  In fact by \cite[page 77]{Gohb} \begin{equation}\label{equivInf} \Phi \sim \Phi_\infty \text{ if and only if } \sup_{n\in\mathbb{N}}\Phi\left(\{\chi_j^{(n)}\}_{j=1}^\infty\right)<\infty\end{equation} where by definition
\[\chi_j^{(n)}= \begin{cases}
    1 &\text{if $j\leq n$} \\
    0 &\text{if $j>n$}.
    \end{cases}\]

\end{itemize}

\section{Main Result and Motivation}

The purpose of this paper is to prove the following Theorem:  \begin{thm}\label{mainThm}
Let $\nu$ be a positive Borel measure on $\mathbb{C}^n$ and $T_\nu$ be the corresponding Toeplitz operator.  Let $0<s\leq 1$, $r>0$ and $\{a_j\}_{j=1}^\infty$ be the $r$-lattice of $\mathbb{C}^n$.  Then for any symmetric norming function  $\Phi$, \[\left(T_\nu\right)^s\in \mathcal{C}_\Phi \text{ if and only if }\Phi\left(\{\left(\widehat{\nu}_r(a_j)\right)^s\}_{j=1}^\infty\right)<\infty.\]
\end{thm}

Toeplitz operators on various spaces and domains have been studied by many different people; see for e.g. \cite{Josh,ToeFock,zhuOp}.  Kehe Zhu and Josh Isralowitz together gave necessary and sufficient conditions for when $T_\nu$ as above is in the Schatten $p$-classes $\mathcal{C}_{\Phi_p}$ for all $p\geq1$ \cite{Josh}.  Specifically, \cite[Theorem 4.4]{Josh}, the following holds:  \begin{thm}
Let $\nu$ be a positive Borel measure on $\mathbb{C}^n$ and $T_\nu$ be the corresponding Toeplitz operator.  Let $p\geq1$, $r>0$ and $\{a_j\}_{j=1}^\infty$ be the $r$-lattice of $\mathbb{C}^n$.  Then \[T_\nu\in \mathcal{C}_{\Phi_p} \text{ if and only if }\Phi_p\left(\{\widehat{\nu}_r(a_j)\}_{j=1}^\infty\right)<\infty.\]
\end{thm}  Since the sets $\mathcal{C}_{\Phi_p}$ are symmetrically-normed ideals for all $p\geq1$, this paper will generalize the above theorem.  We should also point out that one can define Schatten $p$-classes $\mathcal{C}_{\Phi_p}$ for $0<p<1$.  However for such $p$, $\mathcal{C}_{\Phi_p}$ is not a symmetrically-normed ideal, which is why we are not concerned in this paper with such $p$.

The Schatten $p$-classes have been studied by many different people; see for e.g. \cite{JoshGen,Josh,zhuOp,zhuFock}.  Some results that are stated for certain Schatten $p$-classes have been generalized to symmetrically-normed ideals.  For instance Shige Toshi Kuroda in \cite{kuroda} generalized the Weyl von-Neumann theorem to symmetrically-normed ideals $\mathcal{C}_\Phi$ such that the finite rank operators are dense in $\mathcal{C}_\Phi$ and such that $\Phi\nsim \Phi_1$.

Furthermore there are symmetrically normed ideals which are very different from the Schatten $p$-classes.  For instance it is shown in \cite[page 141]{Gohb} how to construct $\Phi$ such that the finite rank operators are not dense in $\mathcal{C}_\Phi$.

Some other symmetrically normed ideals that appear in the literature are the Orlicz ideals, the Marcinkiewicz ideals and the Lorentz ideals; see \cite{Dykema, QXLorentz, XiaOr} for instance.

\section{Important results about symmetric norming functions and symmetrically-normed ideals}\label{snResults}

We use the same notation in this Section as in Subsection \ref{snSection} above.  The next lemma comes from \cite[page 76]{Gohb}:  \begin{lem}\label{snIneq}
For any symmetric norming function $\Phi$ and $\{\xi_j\}_{j=1}^\infty\in \widehat{c}$,
\[\Phi_\infty(\{\xi_j\}_{j=1}^\infty)\leq \Phi\left(\{\xi_j\}_{j=1}^\infty\right)\leq \Phi_1(\{\xi_j\}_{j=1}^\infty)\]
\end{lem}

We also have the following useful result.  \begin{prop}\label{snSub}
Let $\{\lambda_j\}_{j=1}^\infty\in\widehat{c}$ and $\Phi$ be a symmetric norming function.  Suppose $\{p_j\}_{j=1}^\infty \subseteq \mathbb{N}$ such that for each $j$ and $v$ with $j\neq v$, we have $p_j\neq p_v$.  Then $\Phi\left(\{\lambda_{p_j}\}_{j=1}^\infty\right)\leq \Phi\left(\{\lambda_j\}_{j=1}^\infty\right)$.

\end{prop}

\begin{proof}

Let $N\geq1$ and $\Theta$ be a bijection of $\mathbb{N}$ so that $\Theta(j)=p_j$ for all $j=1,\ldots,N$.  Such a $\Theta$ exist since $\{1,2,\ldots,N\}$ and $\{p_1,p_2,\ldots,p_N\}$ have the same cardinality.  Then by Lemma \ref{symm} and the definition of $\Theta$, \[\begin{split}\Phi\left(\lambda_{p_1},\lambda_{p_2},\ldots,\lambda_{p_N},0,\ldots\right)& =\Phi\left(\left|\lambda_{p_1}\right|,\left|\lambda_{p_2}\right|,\ldots,\left|\lambda_{p_N}\right|,0,\ldots\right)\\&= \Phi\left(\left|\lambda_{\Theta(1)}\right|,\left|\lambda_{\Theta(2)}\right|,\ldots,\left|\lambda_{\Theta(N)}\right|,0,\ldots\right)\\&\leq \Phi\left(\{\left|\lambda_{\Theta(j)}\right|\}_{j=1}^\infty\right)= \Phi\left(\{\lambda_j\}_{j=1}^\infty\right).\end{split}\]  Therefore $\Phi\left(\lambda_{p_1},\lambda_{p_2},\ldots,\lambda_{p_N},0,\ldots\right)\leq \Phi\left(\{\lambda_j\}_{j=1}^\infty\right)$.  Since $N\geq1$ was arbitrary, \eqref{extendDef} yields $\Phi\left(\{\lambda_{p_j}\}_{j=1}^\infty\right)\leq \Phi\left(\{\lambda_j\}_{j=1}^\infty\right)$.
\end{proof}

We also have the following result, a proof of which can be easily made using \cite[Theorem 2.1]{Gohb} and the definition of s-numbers.

\begin{prop}\label{idealEquiv}
Let $\Phi$ and $\Psi$ be two symmetric norming functions.  Then \[\mathcal{C}_\Phi=\mathcal{C}_\Psi \text{ if and only if } \Phi \sim\Psi.\]
\end{prop}

\begin{rem}\label{snidealRem}
An immediate consequence of Lemma \ref{eigen} is $|\cdot|_{\Phi_\infty}=\|\cdot\|_{\mathcal{H}}$.  Hence $\mathcal{C}_{\Phi_\infty}=\mathfrak{L}(\mathcal{H})$ and by Proposition \ref{idealEquiv} \begin{equation}\label{snInfinite}\mathcal{C}_\Phi=\mathfrak{L}(\mathcal{H}) \text{ if and only if } \Phi\sim \Phi_\infty. \end{equation}  Furthermore since $\mathcal{H}$ is separable, \eqref{snInfinite} yields \[\mathcal{C}_\Phi\subseteq\mathfrak{LC}(\mathcal{H}) \text{ if and only if } \Phi\nsim \Phi_\infty\]

\end{rem}  From now on, we will use the following notation: for any $g,h\in \mathcal{H}$, let $g\otimes_\mathcal{H} h$ be the operator on $\mathcal{H}$ defined by $(g\otimes_\mathcal{H} h)(v)=\left(\langle v,h\rangle_\mathcal{H} \right)g$ where $\langle\cdot,\cdot\rangle_\mathcal{H}$ is the inner product on $\mathcal{H}\times\mathcal{H}$.  \begin{prop}\label{diag}
Let $\{\lambda_j\}_{j=1}^\infty$be a sequences of complex numbers.  Let $\{e_j\}_{j=1}^\infty$ and $\{b_j\}_{j=1}^\infty$ be two orthonormal sets in $\mathcal{H}$.  If $D=\sum_{j=1}^\infty \lambda_j b_j\otimes_{\mathcal{H}} e_j$ and $D$ is bounded, then for any symmetric norming function $\Phi$, \[|D|_\Phi=\Phi\left(\{\lambda_j\}_{j=1}^\infty\right).\]
\end{prop}

\begin{proof}

By Lemma \ref{eigen}, $|D|_\Phi=|\left(|D|\right)|_\Phi$.  Also by direct calculation $|D|=\sum_{j=1}^\infty |\lambda_j| e_j\otimes_{\mathcal{H}} e_j$.  Because of this, we may assume without loss of generality that $\lambda_j\geq0$ and $b_j=e_j$ for all $j\geq1$.

By definition of $D$, $\sigma_\mathcal{H}(D)=\overline{\{\lambda_j:j\geq 1\}}$ where $\overline{\{\lambda_j:j\geq 1\}}$ is the closure of $\{\lambda_j:j\geq 1\}$ in $[0,\infty)$.  Since by definition of the s-numbers $\{s_j(D):j\geq1\}\subseteq \sigma_\mathcal{H}(D)$, we have $\{s_j(D):j\geq1\}\subseteq \overline{\{\lambda_j:j\geq 1\}}$.

Suppose $\{s_j(D):j\geq1\}\subseteq \{\lambda_j:j\geq1\}$.  Recall from Section \ref{snSection} that the sequence $\{s_j(D)\}_{j=1}^\infty$ is enumerated so as to include multiplicities.  This implies that for each $j\geq1$ there exist $p_j\geq1$ so that $s_j(D)=\lambda_{p_j}(D)$ and so that $j\neq m$ implies $p_j\neq p_m$.  Thus $|D|_\Phi= \Phi\left(\{\lambda_{p_j}\}_{j=1}^\infty\right)$.  By Proposition \ref{snSub}, $\Phi\left(\{\lambda_{p_j}\}_{j=1}^\infty\right)\leq \Phi\left(\{\lambda_j\}_{j=1}^\infty\right)$ and so we have \begin{equation}\label{diagLeft}|D|_\Phi\leq \Phi\left(\{\lambda_j\}_{j=1}^\infty\right).\end{equation}

Assume instead $\{s_j(D):j\geq1\}\nsubseteq\{\lambda_j:j\geq1\}$.  Then by the above statements there exist $j\geq1$ so that $s_j(D)$ is an accumulation point of $ \{\lambda_j:j\geq1\}$.  Let $\epsilon>0$.  We claim that \begin{equation}\label{claimDiag}\begin{split} &\text{for each }j\geq1 \text{ there exist }p_j\geq1 \text{ such that } s_j(D)-\frac{\epsilon}{2^j}\leq\lambda_{p_j} \text{ and so that }m\neq j \text{ implies }\\& p_m\neq p_j.\end{split}\end{equation}  To prove the claim, let $j_0$ be the smallest positive integer such that $s_{j_0}(D)$ is an accumulation point of $\{\lambda_j:j\geq1\}$.

If $j_0=1$, then there exist a subsequence $\{\lambda_{p_j}\}_{j=1}^\infty$ of $\{\lambda_j\}_{j=1}^\infty$ satisfying $s_1(D)-\frac{\epsilon}{2^j}\leq\lambda_{p_j}$ for all $j\geq1$.  Also by definition of the s-numbers, $s_j(D)=s_1(D)$ for all $j\geq 1$.  Thus $s_j(D)-\frac{\epsilon}{2^j}\leq\lambda_{p_j}$ for all $j\geq1$ and \eqref{claimDiag} holds.

So assume $j_0>1$.  Then $\{s_j(D):1\leq j\leq j_0-1\}\subseteq \{\lambda_j:j\geq1\}$.  It follows, as above, that for each $j$ with $1\leq j\leq j_0-1$ there exist $p_j\geq1$ so that $s_j(D)=\lambda_{p_j}$ and so that $1\leq m\leq j_0-1$ with $m\neq j$ implies $p_j\neq p_m$.  Let $v=\max\{p_1,\ldots,p_{j_0-1}\}$.  Since $s_{j_0}(D)$ is an accumulation point of $\{\lambda_j:j\geq1\}$, $s_{j_0}(D)$ is an accumulation point of $\{\lambda_j:j>v\}$.  This means there exist a subsequence $\{\lambda_{n_j}\}_{j=1}^\infty$ of $\{\lambda_j\}_{j=v+1}^\infty$ such that $s_{j_0}(D)-\frac{\epsilon}{2^{j_0+j-1}}\leq \lambda_{n_j}$ for all $j\geq 1$.  Again by definition of the s-numbers, $s_j(D)=s_{j_0}(D)$ for all $j\geq j_0$.  Thus setting $p_{j_0-1+d}=n_d$ for all $d\geq1$, the above work yields $s_j(D)-\frac{\epsilon}{2^j}\leq\lambda_{p_j}$ for all $j\geq1$ and again \eqref{claimDiag} holds.

By \eqref{claimDiag}, the triangle inequality for $\Phi$, Lemma \ref{symm}, Lemma \ref{snIneq} and Proposition \ref{snSub} \[\begin{split}|D|_\Phi-\epsilon\sum_{j=1}^\infty\left(\frac{1}{2}\right)^j\leq |D|_\Phi-\epsilon\Phi\left(\left\{\left(\frac{1}{2}\right)^j\right\}_{j=1}^\infty\right)\leq \Phi\left(\left\{s_j(D)-\frac{\epsilon}{2^j}\right\}_{j=1}^\infty\right)&\leq\Phi\left(\{\lambda_{p_j}\}_{j=1}^\infty\right)\\&\leq \Phi\left(\{\lambda_j\}_{j=1}^\infty\right).\end{split}\]  Since $1=\sum_{j=1}^\infty\left(\frac{1}{2}\right)^j$ and $\epsilon>0$ was arbitrary, we again have \eqref{diagLeft}.

Let $N\geq1$ and $P_N=\sum_{j=1}^N \lambda_je_j\otimes_{\mathcal{H}} e_j$.  Then $0\leq P_N\leq D$.  It follows from Lemma \ref{symm} and Lemma \ref{eigen} that $|P_N|_\Phi\leq |D|_\Phi$.  Moreover $|P_N|_\Phi=\Phi\left(\lambda_1,\lambda_2,\ldots,\lambda_N,0,\ldots\right)$ by definition of the s-numbers.  This implies $\Phi\left(\lambda_1,\lambda_2,\ldots,\lambda_N,0,\ldots\right)\leq |D|_\Phi$.  Since $N\geq1$ was arbitrary, \eqref{extendDef} implies \[\Phi\left(\{\lambda_j\}_{j=1}^\infty\right)\leq |D|_\Phi.\]  Therefore by the above and \eqref{diagLeft}, $|D|_\Phi=\Phi\left(\{\lambda_j\}_{j=1}^\infty\right).$

\end{proof}

The next result, which is \cite[Theorem 5.1]{Gohb} will be used to prove Proposition \ref{snSum} below.

\begin{lem}\label{GohbSup}
Let $\Phi$ be an arbitrary symmetric norming function such that $\Phi\nsim\Phi_\infty$.  If an operator $A\in\mathfrak{L}(\mathcal{H})$ is the weak limit of a sequence $\{A_n\}_{n=1}^\infty$ from $\mathcal{C}_\Phi$ and if $\sup_n|A_n|_\Phi<\infty$, then $A\in\mathcal{C}_\Phi$ and $|A|_\Phi\leq\sup_n|A_n|_\Phi$.
\end{lem}

\begin{prop}\label{snSum}
Let $0<s\leq1$ and $\Phi$ be a symmetric norming function.  If $\{H_j\}_{j=1}^\infty \subseteq \mathfrak{L}(\mathcal{H})$ then there exists a $C>0$ so that \[\left|\left|\sum_{j=1}^\infty H_j\right|^s\right|_\Phi\leq C2^{1-s}\sum_{j=1}^\infty ||H_j|^s|_\Phi.\]\end{prop}

\begin{proof}
Clearly we may assume $\sum_{j=1}^\infty ||H_j|^s|_\Phi<\infty$.  
Thus by Lemma \ref{snIneq} and the fact that $0<s\leq1$, \begin{equation}\label{sumNorm}\left\|\sum_{j=1}^\infty H_j\right\|_\mathcal{H}^s\leq \left(\sum_{j=1}^\infty \|H_j\|_{\mathcal{H}}\right)^s\leq \sum_{j=1}^\infty \|H_j\|_{\mathcal{H}}^s=\sum_{j=1}^\infty \||H_j|^s\|_{\mathcal{H}}\leq\sum_{j=1}^\infty ||H_j|^s|_\Phi\leq 2^{1-s}\sum_{j=1}^\infty ||H_j|^s|_\Phi.\end{equation}  Thus if $\Phi\sim\Phi_\infty$ \[\left|\left|\sum_{j=1}^\infty H_j\right|^s\right|_\Phi\leq C_1\left\|\sum_{j=1}^\infty H_j\right\|_\mathcal{H}^s\leq C_12^{1-s}\sum_{j=1}^\infty ||H_j|^s|_\Phi\] for some $C_1>0$.

Now assume $\Phi\nsim \Phi_\infty$.  Then by Remark \ref{snidealRem}, $\sum_{j=1}^\infty ||H_j|^s|_\Phi<\infty$ implies $|H_j|^s\in\mathfrak{LC}(\mathcal{H})$ for all $j\geq1$.  So by Proposition \ref{prop} 
$|H_j|\in\mathfrak{LC}(\mathcal{H})$ which itself gives us $H_j\in\mathfrak{LC}(\mathcal{H})$ for every $j\geq1$.  Hence for each $j\geq1$, there exist two orthonormal sets $\left\{\eta_q^{(j)}\right\}_{q=1}^\infty$ and $\left\{\sigma_p^{(j)}\right\}_{p=1}^\infty$ in $\mathcal{H}$ such that $H_j=\sum_{p=1}^\infty s_p(H_j) \eta_p^{(j)}\otimes_{\mathcal{H}} \sigma_p^{(j)}$ and $\lim_{p\rightarrow\infty}s_p(H_j)=0$.

Let $M\geq 1$ and $N\geq1$.  Let $F_N^{(j)}=\sum_{p=1}^N s_p(H_j) \eta_p^{(j)}\otimes_{\mathcal{H}} \sigma_p^{(j)}$.  Then by \cite[Lemma 3.1]{Xia}, \begin{equation}\label{snSumFiniteRank}\left|\left|\sum_{j=1}^M F_N^{(j)}\right|^s\right|_\Phi\leq 2^{1-s}\sum_{j=1}^M \left|\left|F_N^{(j)}\right|^s\right|_\Phi.\end{equation}  It is also true by definition of the s-numbers and Proposition \ref{prop} 
that \[||F_N^{(j)}|^s|_\Phi=\Phi\left(|s_1\left(H_j\right)|^s,|s_2\left(H_j\right)|^s,\ldots,|s_N\left(H_j\right)|^s,0,\ldots\right)\] for all $j\geq1$.  It follows from Lemma \ref{symm} that \begin{equation}\label{finiteRankIneq}\left|\left|F_N^{(j)}\right|^s\right|_\Phi\leq \left|\left|H_j\right|^s\right|_\Phi\end{equation} for all $j\geq1$.  Then combining \eqref{snSumFiniteRank} and \eqref{finiteRankIneq} gives us \begin{equation}\label{finiteRank}\left|\left|\sum_{j=1}^M F_N^{(j)}\right|^s\right|_\Phi\leq 2^{1-s}\sum_{j=1}^M \left|\left|H_j\right|^s\right|_\Phi\leq 2^{1-s}\sum_{j=1}^\infty \left|\left|H_j\right|^s\right|_\Phi \text{ for all }N\geq 1.\end{equation}

Furthermore since $\lim_{p\rightarrow\infty}s_p(H_j)=0$ for all $j\geq1$, $\{F_N^{(j)}\}_{N=1}^\infty$ converges to $H_j$ with respect to the norm topology.  Then the Stone-Weierstrass Theorem and the continuous functional calculus \cite[page 62]{zhuAl} yields $\left\{\left|\sum_{j=1}^M F_N^{(j)}\right|^s\right\}_{N=1}^\infty$ converges to $\left|\sum_{j=1}^M H_j\right|^s$ with respect to the norm topology.  Then by \eqref{finiteRank} and Lemma \ref{GohbSup}, \[\left|\left(\sum_{j=1}^M H_j\right)^s\right|_\Phi\leq 2^{1-s}\sum_{j=1}^\infty ||H_j|^s|_\Phi \text{ for any }M\geq1.\]  Moreover since $\sum_{j=1}^\infty ||H_j|^s|_\Phi<\infty$, \eqref{sumNorm} implies $\sum_{j=1}^\infty\|H_j\|_{\mathcal{H}}<\infty$.  It follows that\\ \noindent $\left\{\sum_{j=1}^M H_j\right\}_{M=1}^\infty$ converges to $\sum_{j=1}^\infty H_j$ in the norm topology.  Thus another application of the continuous functional calculus \cite[page 62]{zhuAl} and the Stone-Weierstrass theorem shows $\left\{\left|\sum_{j=1}^M H_j\right|^s\right\}_{M=1}^\infty$ converges to $\left|\sum_{j=1}^\infty H_j\right|^s$ in the norm topology.  Then by another application Lemma \ref{GohbSup}, \[\left|\left|\sum_{j=1}^\infty H_j\right|^s\right|_\Phi\leq 2^{1-s}\sum_{j=1}^\infty ||H_j|^s|_\Phi.\]  Therefore with $C=\max\{C_1,1\}$, we have \[\left|\left|\sum_{j=1}^\infty H_j\right|^s\right|_\Phi\leq C2^{1-s}\sum_{j=1}^\infty ||H_j|^s|_\Phi.\]

\end{proof}

\section{Proof of Main Result}\label{mainCh}

From now on $\otimes=\otimes_{H^2} $ where $\otimes_\mathcal{H}$ is defined in Section \ref{snResults} for any separable complex Hilbert space $\mathcal{H}$.

\subsection{An Important Theorem}
The next theorem will be very useful in proving Theorem \ref{mainThm}.

\begin{thm}\label{Berezin}
Let $\rho>0$ $\gamma>0$, $\alpha>0$ $0<s\leq 1$, and $\{b_j\}_{j=1}^\infty$ be the $\rho$-lattice of $\mathbb{C}^n$.  Let $\nu$ be a positive Borel measure on $\mathbb{C}^n$ and $\Phi$ be a symmetric norming function.  Then \[e^{-s\gamma^2\alpha(2n)}\Phi\left(\{\left(\widehat{\nu}_\gamma(b_j)\right)^s\}_{j=1}^\infty\right)\leq \Phi\left(\{\left(\widetilde{\nu}_\alpha(b_j)\right)^s\}_{j=1}^\infty\right)\leq C\Phi\left(\{\left(\widehat{\nu}_\rho(b_j)\right)^s\}_{j=1}^\infty\right)\] where $C=e^{s\alpha4n^2\rho^2}\sum_{m=0}^\infty (2m+1)^{2n}e^{-s\alpha\rho^2\left(m-2n\right)^2}$.
\end{thm}

\begin{proof}
 By Lemma \ref{tilde} and Lemma \ref{symm}, \[e^{-s\gamma^2\alpha(2n)}\Phi\left(\{\left(\widehat{\nu}_\gamma(b_j)\right)^s\}_{j=1}^\infty\right)\leq \Phi\left(\{\left(\widetilde{\nu}_\alpha(b_j)\right)^s\}_{j=1}^\infty\right).\]

For any $z\in\mathbb{C}^n$ and any $\mathcal{A}\subseteq \mathbb{C}^n$, let $\mathcal{A}+\{z\}=\{a+z:a\in \mathcal{A}\}$.  Then by the definition of a $\rho$-lattice as in Section \ref{snnotat}, \begin{equation}\label{latticeBere}\{b_q\}_{q=1}^\infty=\{b_q\}_{q=1}^\infty+\{b_j\} \text{ for all }j\geq1.\end{equation}  From the above and \eqref{lattice} it follows that $\mathbb{C}^n=\bigcup_{q=1}^\infty B(b_j+b_q,\rho)$ for every $j\geq1$.  Hence \begin{equation}\label{leftBerNu}\begin{split}\left(\widetilde{\nu}_\alpha(b_j)\right)^s= \left(\int_{\mathbb{C}^n}e^{-\alpha|z-b_j|^2}d\nu(z)\right)^s&\leq \left(\sum_{q=1}^\infty \int_{B(b_q+b_j,\rho)}e^{-\alpha|z-b_j|^2}d\nu(z)\right)^s
\\&\leq \sum_{q=1}^\infty \left(\int_{B(b_q+b_j,\rho)}e^{-\alpha|z-b_j|^2}d\nu(z)\right)^s
\end{split}\end{equation}  For any $z\in B(b_q+b_j,\rho)$, we have by \eqref{normIn} \[|z-(b_q+b_j)|\leq \sqrt{2n}|z-(b_q+b_j)|_\infty<\rho\sqrt{2n}.\]  Then from the triangle inequality \[\begin{split}|z-b_j|^2\geq (|z-(b_j+b_q)|-|b_q|)^2&\geq |b_q|^2-2|z-(b_j+b_q)||b_q|\\&\geq |b_q|^2-2\rho\sqrt{2n}|b_q|.\end{split}\]  This implies \begin{equation}\label{rightBerNu}\begin{split}\sum_{q=1}^\infty \left(\int_{B(b_q+b_j,\rho)}e^{-\alpha|z-b_j|^2}d\nu(z)\right)^s&\leq \sum_{q=1}^\infty \left(\int_{B(b_q+b_j,\rho)}e^{-\alpha|b_q|^2+2\alpha\rho\sqrt{2n}|b_q|}d\nu(z)\right)^s\\&= \sum _{q=1}^\infty e^{-s\alpha|b_q|^2+2s\alpha\rho\sqrt{2n}|b_q|}\left(\widehat{\nu}_\rho(b_q+b_j)\right)^s.\end{split}\end{equation}

So by Lemma \ref{symm}, \eqref{leftBerNu} and \eqref{rightBerNu}, \[\Phi\left(\left\{\left(\widetilde{\nu}_\alpha(b_j)\right)^s\right\}_{j=1}^\infty\right)\leq \Phi\left(\left\{\sum _{q=1}^\infty e^{-s\alpha|b_q|^2+2s\alpha\rho\sqrt{2n}|b_q|}\left(\widehat{\nu}_\rho(b_q+b_j)\right)^s\right\}_{j=1}^\infty\right).\]

Then by the triangle inequality for $\Phi$ and \eqref{latticeBere}, \[\begin{split}&\Phi\left(\left\{\sum _{q=1}^\infty e^{-s\alpha|b_q|^2+2s\alpha\rho\sqrt{2n}|b_q|}\left(\widehat{\nu}_\rho(b_q+b_j)\right)^s\right\}_{j=1}^\infty\right)\\&\leq \sum_{q=1}^\infty \Phi\left(e^{-s\alpha|b_q|^2+2s\alpha\rho\sqrt{2n}|b_q|}\{\left(\widehat{\nu}_\rho(b_q+b_j)\right)^s\}_{j=1}^\infty\right)\\&=\left[\sum_{q=1}^\infty e^{-s\alpha|b_q|^2+2s\alpha\rho\sqrt{2n}|b_q|}\Phi\left(\{\left(\widehat{\nu}_\rho(b_q+b_j)\right)^s\}_{j=1}^\infty\right)\right]\\&=
\left[\sum_{q=1}^\infty e^{-s\alpha|b_q|^2+2s\alpha\rho\sqrt{2n}|b_q|}\right]\Phi\left(\{\left(\widehat{\nu}_{\rho}(b_j)\right)^s\}_{j=1}^\infty\right). \end{split}\]

By \eqref{normIn}, \[\begin{split}\sum_{q=1}^\infty e^{-s\alpha|b_q|^2+2s\alpha\rho\sqrt{2n}|b_q|}\leq \sum_{q=1}^\infty e^{-s\alpha|b_q|_\infty^2+s\alpha4n\rho|b_q|_\infty}&=\sum_{q=1}^\infty e^{-s\alpha|b_q|_\infty^2+s\alpha4n\rho|b_q|_\infty-s\alpha4n^2\rho^2+s\alpha4n^2\rho^2}\\&=\sum_{q=1}^\infty e^{-s\alpha\left(|b_q|_\infty-2n\rho\right)^2+s\alpha4n^2\rho^2}\\&= e^{s\alpha4n^2\rho^2}\sum_{q=1}^\infty e^{-s\alpha\left(|b_q|_\infty-2n\rho\right)^2}.\end{split}\]

By definition of a $\rho$-lattice, we have $b_q=\rho(j_1^{(q)}+il_1^{(q)},\ldots,j_{n}^{(q)}+il_n^{(q)})$ for some integers $j_1^{(q)},l_1^{(q)},\ldots,j_{n}^{(q)},\ldots,l_n^{(q)}$.  Hence $|b_q|_\infty=\rho\max\left\{\left|j_d^{(q)}\right|,\left|l_d^{(q)}\right|\right\}_{d=1}^n$.  Thus for any $m\geq0$, $|b_q|_\infty \leq \rho m$ if and only if $|j_d^{(q)}|\leq m$ and $|l_d^{(q)}|\leq m$ for each $1\leq d\leq n$.  Moreover the number of integers $j$ satisfying $|j|\leq m$ is $2m+1$.  This means the cardinality of the set $\{q\in\mathbb{N}: |b_q|_\infty = m\rho\}$ is not larger than $(2m+1)^{2n}$.  So with $B_m=\{q\in\mathbb{N}:|b_q|_\infty=\rho m\}$ the above statements tells us \[\begin{split}e^{s\alpha4n^2\rho^2}\sum_{q=1}^\infty e^{-s\alpha\left(|b_q|_\infty-2n\rho\right)^2}&=e^{s\alpha4n^2\rho^2}\sum_{m=0}^\infty \sum_{q\in B_m}e^{-s\alpha\left(|b_q|_\infty-2n\rho\right)^2}\\&\leq e^{s\alpha4n^2\rho^2}\sum_{m=0}^\infty (2m+1)^{2n}e^{-s\alpha\rho^2\left(m-2n\right)^2}.\end{split}\] 

So with $C=e^{s\alpha4n^2\rho^2}\sum_{m=0}^\infty (2m+1)^{2n}e^{-s\alpha\rho^2\left(m-2n\right)^2}$, \[\Phi\left(\left\{\left(\widetilde{\nu}_\alpha(b_j)\right)^s\right\}_{j=1}^\infty\right)\leq C\Phi\left(\{\left(\widehat{\nu}_{\rho}(b_j)\right)^s\}_{j=1}^\infty\right).\]  Therefore \[e^{-s\alpha\gamma^2(2n)}\Phi\left(\{\left(\widehat{\nu}_{\gamma}(b_j)\right)^s\}_{j=1}^\infty\right)\leq \Phi\left(\left\{\left(\widetilde{\nu}_\alpha(b_j)\right)^s\right\}_{j=1}^\infty\right)\leq C\Phi\left(\{\left(\widehat{\nu}_{\rho}(b_j)\right)^s\}_{j=1}^\infty\right).\]

\end{proof}

\subsection{Proof of Sufficiency}\label{sufficiency}


We will show sufficiency of Theorem \ref{mainThm} here by proving \[|\left(T_\nu\right)^s|_\Phi\leq C\Phi\left(\{\left(\widehat{\nu}_r(a_j)\right)^s\}_{j=1}^\infty\right)\] for some constant $C=C(r,s)>0$.

If $\Phi\left(\{\left(\widehat{\nu}_r(a_j)\right)^s\}_{j=1}^\infty\right)=\infty$, then the above clearly holds.  So assume $\Phi\left(\{\left(\widehat{\nu}_r(a_j)\right)^s\}_{j=1}^\infty\right)<\infty$.  It follows from Lemma \ref{snIneq} that $\Phi_\infty\left(\{\left(\widehat{\nu}_r(a_j)\right)^s\}_{j=1}^\infty\right)<\infty$, which further implies $\|T_\nu\|<\infty$ \cite{Josh}.  Let $\delta$ be as in Theorem \ref{atomCor} and choose $M\in\mathbb{N}$ with $M>2$ and such that $\frac{r}{M}<\delta$.  For ease of notation, let $\rho=\frac{r}{M}$.  Let $\{b_j\}_{j=1}^\infty$ be the $\rho$-lattice of $\mathbb{C}^n$, $\{e_{b_j}\}_{j=1}^\infty$ be an orthonormal set in $H^2$ and let $B^{(\rho)}$ be the operator on $H^2$ defined by \[B^{(\rho)}=\sum_{j=1}^\infty e_{b_j}\otimes k_{b_j}.\]  Then for any $f,g\in H^2$ we have by Bessel's inequality, the Cauchy-Schwarz inequality and \cite[Lemma 2.1]{Josh}, \begin{equation}\label{Bbound}\begin{split}|\langle B^{(\rho)}g,f\rangle|& \leq \sum_{j=1}^\infty |\langle f, e_{b_j}\rangle||\langle k_{b_j},g\rangle|\\&=\sum_{j=1}^\infty |\langle f, e_{b_j}\rangle||g(b_j)|e^{\frac{-|b_j|^2}{2}}\\&\leq \left(\sum_{j=1}^\infty|\langle f,e_{b_j}\rangle|^2\right)^{\frac{1}{2}}\left(\sum_{j=1}^\infty |g(b_j)|^2e^{-|b_j|^2}\right)^{\frac{1}{2}}\\&\leq \|f\|_2J\left(\sum_{j=1}^\infty\int_{B(b_j,\rho)}|g(z)|^2e^{-|z|^2}dV(z)\right)^{\frac{1}{2}}\\&\leq J q\|f\|_2\|g\|_2\end{split}\end{equation} where $J=J(\rho,n)>0$ and where $q\in\mathbb{N}$ is such that every $z\in\mathbb{C}^n$ lies in at most $q$ of the sets $B(b_j,\rho)$.  Such an $q$ clearly exists.  Since $f,g\in H^2$ were arbitrary, \eqref{Bbound} implies $\|B^{(\rho)}\|\leq q J$ and $B^{(\rho)}\in\mathfrak{L}(H^2)$.

Let \begin{equation}\label{ToeInv1}\Gamma^{(\rho)}=\left(B^{(\rho)}\right)^*B^{(\rho)}.\end{equation}  This means $\langle\Gamma^{(\rho)}g,g\rangle=\langle B^{(\rho)}g,B^{(\rho)}g\rangle$ for all $g\in H^2$.  Hence \[\langle\Gamma^{(\rho)}g,g\rangle=\sum_{j=1}^\infty \left|\langle g,k_{b_j}\rangle\right|^2.\]  So $\Gamma^{\rho}\geq0$ and by Theorem \ref{atomCor}, there exist a constant $W_1(\rho)>0$ so that \[W_1(\rho)\|g\|_2^2\leq \langle\Gamma^{(\rho)}g,g\rangle \text{ for all } g\in H^2.\]  Thus by \cite[Corollary 4.9]{Doug}, $\Gamma^{(\rho)}$ is invertible on $H^2$.  So \eqref{ToeInv1} implies \begin{equation}\label{threeProdToe}T_\nu=\Gamma^{{(\rho)}^{-1}}\Gamma^{(\rho)}T_\nu\Gamma^{(\rho)}\Gamma^{{(\rho)}^{-1}}=\Gamma^{{(\rho)}^{-1}}\left(B^{(\rho)}\right)^*B^{(\rho)} T_\nu\left(B^{(\rho)}\right)^*B^{(\rho)}\Gamma^{{(\rho)}^{-1}}.\end{equation}

We claim that,  for some constant $C_1=C_1(\rho,n,s)>0$, \begin{equation}\label{ToePhi}\left|\left(T_\nu\right)^s\right|_\Phi\leq C_1\left|\left(B^{(\rho)} T_\nu\left(B^{(\rho)}\right)^*\right)^s\right|_\Phi.\end{equation}  To prove this claim, first note if $|\left(B^{(\rho)} T_\nu\left(B^{(\rho)}\right)^*\right)^s|_\Phi=\infty$, then the above clearly holds.  So assume instead $|\left(B^{(\rho)} T_\nu\left(B^{(\rho)}\right)^*\right)^s|_\Phi<\infty$.  If $\Phi\nsim\Phi_\infty$, then by Remark \ref{snidealRem}, Lemma \ref{eigen} and Proposition \ref{prop} yields \[\begin{split}&\left|\left(\Gamma^{{(\rho)}^{-1}}\left(B^{(\rho)}\right)^*B^{(\rho)} T_\nu\left(B^{(\rho)}\right)^*B^{(\rho)}\Gamma^{{(\rho)}^{-1}}\right)^s\right|_\Phi \\&=\Phi\left(\left\{s_j\left(\Gamma^{{(\rho)}^{-1}}\left(B^{(\rho)}\right)^*B^{(\rho)} T_\nu\left(B^{(\rho)}\right)^*B^{(\rho)}\Gamma^{{(\rho)}^{-1}}\right)^s\right\}_{j=1}^\infty\right)\\&\leq \|\Gamma^{{(\rho)}^{-1}}\left(B^{(\rho)}\right)^*\|^{2s}\Phi\left(\{s_j(B^{(\rho)}T_\nu \left(B^{(\rho)}\right)^*)^s\}_{j=1}^\infty\right)\\&=\|\Gamma^{{(\rho)}^{-1}}\left(B^{(\rho)}\right)^*\|^{2s}\left|\left(B^{(\rho)} T_\nu\left(B^{(\rho)}\right)^*\right)^s\right|_\Phi\end{split}\]  If instead $\Phi\sim\Phi_\infty$ then for some $V>0$ that only depends on $\Phi$, \[\begin{split}&\left|\left(\Gamma^{{(\rho)}^{-1}}\left(B^{(\rho)}\right)^*B^{(\rho)} T_\nu\left(B^{(\rho)}\right)^*B^{(\rho)}\Gamma^{{(\rho)}^{-1}}\right)^s\right|_\Phi\\&\leq V\left\|\left(\Gamma^{{(\rho)}^{-1}}\left(B^{(\rho)}\right)^*B^{(\rho)} T_\nu\left(B^{(\rho)}\right)^*B^{(\rho)}\Gamma^{{(\rho)}^{-1}}\right)^s\right\|\\&=V\left\|\Gamma^{{(\rho)}^{-1}}\left(B^{(\rho)}\right)^*B^{(\rho)} T_\nu\left(B^{(\rho)}\right)^*B^{(\rho)}\Gamma^{{(\rho)}^{-1}}\right\|^s\\&\leq V\left\|\Gamma^{{(\rho)}^{-1}}\left(B^{(\rho)}\right)^*\right\|^{2s}\left\|\left(B^{(\rho)} T_\nu\left(B^{(\rho)}\right)^*\right)^s\right\|\\&\leq V\left\|\Gamma^{{(\rho)}^{-1}}\left(B^{(\rho)}\right)^*\right\|^{2s}\left|\left(B^{(\rho)} T_\nu\left(B^{(\rho)}\right)^*\right)^s\right|_\Phi\end{split}\] where the latter inequality follows by Lemma \ref{snIneq}.  Thus \eqref{ToePhi} holds.

By direct calculation and \cite[page 188-191]{doubleSeq} \begin{equation}\label{bSumb}B^{(\rho)}T_\nu \left(B^{(\rho)}\right)^*=\sum_{j=1}^\infty \sum_{i=1}^\infty \langle T_\nu k_{b_i},k_{b_i+b_j}\rangle e_{b_i+b_j}\otimes e_{b_i}.\end{equation}  So with $H_j$ defined by \begin{equation}\label{hjdef}H_j=\sum_{i=1}^\infty \langle T_\nu k_{b_i},k_{b_i+b_j}\rangle e_{b_i+b_j}\otimes e_{b_i},\end{equation} we have \begin{equation}\label{Bsum}B^{(\rho)}T_\nu \left(B^{(\rho)}\right)^*=\sum_{j=1}^\infty H_j\end{equation} and by Proposition \ref{snSum} \begin{equation}\label{sumH}\left|\left|\sum_{j=1}^\infty H_j \right|^s\right|_\Phi \leq 2^{1-s}D_1\sum_{j=1}^\infty ||H_j|^s|_\Phi\end{equation} for some constant $D_1>0$.  Combining \eqref{ToePhi}, \eqref{Bsum} and \eqref{sumH} together gives us \begin{equation}\label{sumToe}|\left(T_\nu\right)^s|_\Phi\leq 2^{1-s}C_2\sum_{j=1}^\infty ||H_j|^s|_\Phi\end{equation} where $C_2= C_1D_1$.  By the definition of $H_j$ and Proposition \ref{diag}, \[ ||H_j|^s|_\Phi=\Phi\left(\left\{|\langle T_\nu k_{b_i},k_{b_i+b_j}\rangle|^s\right\}_{i=1}^\infty\right).\]

Let $i\geq 1$ and $j\geq 1$.  Since $\|T_\nu\|<\infty$ \cite[Lemma 4.1]{ToeFock} implies \begin{equation}\label{suffElement}\begin{split}\left|\langle T_\nu k_{b_i},k_{b_i+b_j}\rangle\right|= \left|\int_{\mathbb{C}^n} k_{b_i}(z)\overline{k_{b_i+b_j}(z)}e^{-|z|^2}d\nu(z)\right|&\leq \int_{\mathbb{C}^n} \left|k_{b_i}(z)\overline{k_{b_i+b_j}(z)}e^{-|z|^2}\right|d\nu(z)\\ &=\int_{\mathbb{C}^n} e^{\frac{-|z-b_i|^2}{2}}e^{\frac{-|z-(b_i+b_j)|^2}{2}}d\nu(z)\end{split}\end{equation} where the latter equality comes from direct calculation.

Now by the triangle inequality, \[\frac{|b_j|}{2}\leq|z-b_i| \text{ or } \frac{|b_j|}{2}\leq |z-(b_i+b_j)| \text{ for all } z\in\mathbb{C}^n.\]  If $\frac{|b_j|}{2}\leq |z-(b_i+b_j)|$, then $\frac{|b_j|^2}{8}+\frac{|z-b_i|^2}{2}\leq \frac{|z-(b_i+b_j)|^2}{2}+\frac{|z-b_i|^2}{2}$.  Thus $\frac{|b_j|^2}{8}+\frac{|z-b_i|^2}{2}+\frac{|z-(b_i+b_j)|^2}{2}\leq  |z-(b_i+b_j)|^2+|z-b_i|^2$.  Likewise if $\frac{|b_j|}{2}\leq|z-b_i|$ then $\frac{|b_j|^2}{8}+\frac{|z-b_i|^2}{2}+\frac{|z-(b_i+b_j)|^2}{2}\leq |z-(b_i+b_j)|^2+|z-b_i|^2$.  In either case, \begin{equation}\label{suffElement2}e^{\frac{-|z-b_i|^2}{2}}e^{\frac{-|z-(b_i+b_j)|^2}{2}}\leq e^{\frac{-|b_j|^2}{16}} e^{\frac{-|z-b_i|^2}{4}}e^{\frac{-|z-(b_i+b_j)|^2}{4}}.\end{equation} From \eqref{suffElement} and \eqref{suffElement2} we obtain \begin{equation}\label{boundCpt}\left|\langle T_\nu k_{b_i},k_{b_i+b_j}\rangle\right|^s\leq e^{\frac{-s|b_j|^2}{16}}\left(\widetilde{\nu}_{\frac{1}{4}}(b_i)\right)^s \text{ for every }i\geq1 \text{ and } j\geq1.\end{equation}  Hence by Lemma \ref{symm}, \begin{equation}\label{HnormBound}||H_j|^s|_\Phi\leq e^{\frac{-s|b_j|^2}{16}}\Phi\left(\left\{\left(\widetilde{\nu}_{\frac{1}{4}}(b_i)\right)^s\right\}_{i=1}^\infty\right) \text{ for all }j\geq1.\end{equation}  Thus by \eqref{sumToe}, \[|\left(T_\nu\right)^s|_\Phi\leq 2^{1-s}C_2\Phi\left(\left\{\left(\widetilde{\nu}_{\frac{1}{4}}(b_i)\right)^s\right\}_{i=1}^\infty\right)\sum_{j=1}^\infty e^{\frac{-s|b_j|^2}{16}}.\]

By \eqref{normIn} and an argument similar to one used in the proof of Theorem \ref{Berezin} shows \[\sum_{j=1}^\infty e^{\frac{-s|b_j|^2}{16}}\leq\sum_{j=1}^\infty e^{\frac{-s|b_j|_\infty^2}{16}}\leq \sum_{j=1}^\infty (2j+1)^{2n}e^{\frac{-s(\rho j)^2}{16}}.\]  So in fact \[|\left(T_\nu\right)^s|_\Phi\leq 2^{1-s}C_3\Phi\left(\left\{\left(\widetilde{\nu}_{\frac{1}{4}}(b_i)\right)^s\right\}_{i=1}^\infty\right)\] where $C_3=C_2\sum_{j=1}^\infty e^{\frac{-s|b_j|^2}{16}}$.  Then by Theorem \ref{Berezin} \[|\left(T_\nu\right)^s|_\Phi\leq C_4\Phi\left(\{\left(\widehat{\nu}_\rho(b_j)\right)^s\}_{j=1}^\infty\right)\] for some positive constant $C_4$.

For any $j\geq1$, \eqref{lattice} and the definition of $\rho$ implies there exist a unique $p_j\geq1$ so that $a_{p_j}\in\overline{B\left(b_j,\frac{\rho}{2}\right)}$. Assume $|b_j-a_{p_j}|_\infty=\frac{\rho}{2}$.  It follows from the definition of an $r$-lattice and the definition of $\rho$ that $|Mw-d|=\frac{1}{2}$ for some $w,d\in\mathbb{Z}$.  Thus $\frac{1}{2}\in\mathbb{Z}$, which is a contradiction.  So we must have $a_{p_j}\in B\left(b_j,\frac{\rho}{2}\right)$.  Now if $p_j=p_q$ for some $q\neq j$, then by the triangle inequality, $|b_j-b_q|_\infty\leq |a_{p_j}-b_j|_\infty+|a_{p_j}-b_q|_\infty<\rho$.  This is also a contradiction since $|b_q-b_j|_\infty\geq \rho$.  Thus $j\neq q$ implies $p_j\neq p_q$.  We also have $B(b_j,\rho)\subseteq B(a_{p_j},\frac{3\rho}{2})$.  It follows that $\left(\widehat{\nu}_{\rho}(b_j)\right)^s\leq \left(\widehat{\nu}_r(a_{p_j})\right)^s$.  Hence by Lemma \ref{symm}, \[\Phi\left(\{\left(\widehat{\nu}_{\rho}(b_j)\right)^s\}_{j=1}^\infty\right)\leq \Phi\left(\{\left(\widehat{\nu}_r(a_{p_j})\right)^s\}_{j=1}^\infty\right).\]

It is also true by Proposition \ref{snSub} and the above statements that \[\Phi\left(\{\left(\widehat{\nu}_r(a_{p_j})\right)^s\}_{j=1}^\infty\right)\leq \Phi\left(\{\left(\widehat{\nu}_r(a_j)\right)^s\}_{j=1}^\infty\right).\]  Therefore with $C=C_4$, \[|\left(T_\nu\right)^s|_\Phi\leq C\Phi\left(\{\left(\widehat{\nu}_r(a_j)\right)^s\}_{j=1}^\infty\right).\]

\subsection{Proof of Necessity}\label{necessary}

We will show necessity of Theorem \ref{mainThm} here by proving \[\Phi\left(\{\left(\widehat{\nu}_r(a_j)\right)^s\}_{j=1}^\infty\right)\leq V_1|\left(T_\nu\right)^s|_\Phi\] for some constant $V_1=V_1(r,s)>0$.

If $|\left(T_\nu\right)^s|_\Phi=\infty$ then the above clearly holds.  So assume $|\left(T_\nu\right)^s|_\Phi<\infty$.  Then by Remark \ref{snidealRem}, $\left(T_\nu\right)^s\in\mathfrak{L}(H^2)$.  This implies $\left(\left(T_\nu\right)^s\right)^{\frac{1}{s}}\in\mathfrak{L}(H^2)$.  Moreover by definition, $\left(\left(T_\nu\right)^s\right)^{\frac{1}{s}} = T_\nu$.  Thus $T_\nu\in\mathfrak{L}(H^2)$.

Fix $m\in\mathbb{N}$ such that $m\geq2$ and let $R=mr$.  Let $A_m=\{(j_1+iv_1,\ldots, j_n+iv_n):\{j_q\}_{q=1}^n,\{v_q\}_{q=1}^n\subseteq \{0,1,2,\ldots,m-1\}\}$.  For convenience, we write $A_m=\{y_j\}_{j=1}^{m^{2n}}$.  For each $y_j\in A_m$, let $\Gamma_{y_j}$ be the subset of $\{a_j\}_{j=1}^\infty$ defined by $\Gamma_{y_j}=\{Ry+ry_j:y=(u_1+iw_1,\ldots,u_n+iw_n)\text{ for some }\{u_j\}_{j=1}^n,\{w_j\}_{j=1}^n\subseteq \mathbb{Z}\}$.  It follows that, for any $j\geq1$, if $\omega,\gamma\in \Gamma_{y_j}$ with $\omega\neq\gamma$ then $|\omega-\gamma|_\infty\geq R$.  It is also easy to show that \begin{equation}\label{LatticePartition}\begin{split}&\{a_j\}_{j=1}^\infty=\bigcup_{p=1}^{m^{2n}}\Gamma_{y_p} \text{ and } \Gamma_{y_1},\Gamma_{y_2},\ldots, \Gamma_{y_{m^{2n}}}\text{ are pairwise disjoint}.\end{split}\end{equation}

Now fix $p$ and write $\Gamma_{y_p}$ as $\{\rho_j\}_{j=1}^\infty$ for convenience.  Let $\{e_{\rho_j}\}_{j=1}^\infty$ be an orthonormal set in $H^2$ and $A$ be the operator defined on $H^2$ by $A=\sum_{j=1}^\infty k_{\rho_j}\otimes e_{\rho_j}.$  Then by a calculation almost identical to \eqref{Bbound}, $\|A\|\leq J$ for some constant $J=J(r,n)>0$ and $A\in\mathfrak{L}(H^2)$.  Let $\{v_q\}_{q=1}^\infty=\Gamma_{(0,\ldots,0)}$ and $\{w_q\}_{q=1}^\infty=\Gamma_{(0,\ldots,0)}\setminus\{(0,\ldots,0)\}$ where $(0,\ldots,0)$ is the origin of $\mathbb{C}^n$.  By definition of $\{\rho_j\}_{j=1}^\infty$, we have \[\{\rho_d+v_q\}_{q=1}^\infty=\{\rho_j\}_{j=1}^\infty \text{ for any }d\geq1.\]  Using the above, \cite[page 188-191]{doubleSeq}, and direct calculation, we can write $A^*T_\nu A$ as $A^*T_\nu A=D+E$ where \[D=\sum_j \langle T_\nu k_{\rho_j},k_{\rho_j}\rangle e_{\rho_j}\otimes e_{\rho_j} \text{ and }E=\sum_{q=1}^\infty\sum_{j=1}^\infty \langle T_\nu k_{\rho_j},k_{\rho_j+w_q}\rangle e_{\rho_j+w_q}\otimes e_{\rho_j}.\]  Then from Proposition \ref{snSum}, \[||D|^s|_\Phi\leq W2^{1-s}\left(|(A^*T_\nu A)^s|_\Phi+||E|^s|_\Phi\right)\] for some $W>0$ that is independent of $m$ and $\{\rho_j\}_{j=1}^\infty$.  A calculation similar to the one used to verify \eqref{ToePhi} shows $|(A^*T_\nu A)^s|_\Phi\leq C_1|\left(T_\nu\right)^s|_\Phi$ where $C_1>0$ also does not depend on $m$ or $\{\rho_j\}_{j=1}^\infty$.  This implies \begin{equation}\label{TDE}||D|^s|_\Phi\leq 2^{1-s}WC_1|\left(T_\nu\right)^s|_\Phi+W2^{1-s}||E|^s|_\Phi.\end{equation}

By Proposition \ref{diag}, $||D|^s|_\Phi=\Phi\left(\{|\langle T_\nu k_{\rho_j},k_{\rho_j}\rangle|^s\}_{j=1}^\infty\right)$.  Since $T_\nu$ is bounded, $\langle T_\nu(k_z),k_z\rangle = \widetilde{\nu}_1(z)$ for all $z\in\mathbb{C}^n$ \cite{Josh}.  Thus $||D|^s|_\Phi=\Phi\left(\{\left(\widetilde{\nu}_1(\rho_j)\right)^s\}_{j=1}^\infty\right)$ and by Theorem \ref{Berezin}, \begin{equation}\label{TDE1}e^{-sr^2(2n)}\Phi\left(\{\left(\widehat{\nu}_r(\rho_j)\right)^s\}_{j=1}^\infty\right)\leq ||D|^s|_\Phi.\end{equation}

Note that for each $q\geq1$, $|w_q|\geq R$ by \eqref{normIn}.  So using calculations similar to those used to derive \eqref{suffElement2} shows \begin{equation}\label{calcE}e^{\frac{-|z-\rho_j|^2}{2}}e^{\frac{-|z-(\rho_j+w_q)|^2}{2}}\leq e^{\frac{-R^2}{16}}e^{\frac{-|w_q|^2}{32}}e^{\frac{-|z-\rho_j|^2}{8}}e^{\frac{-|z-(\rho_j+w_q)|^2}{8}}\text{ for all }j\geq 1\text{ and }q\geq1.\end{equation}  

For each $q\geq1$, let \[E_q=\sum_{j=1}^\infty \langle T_\nu k_{\rho_j},k_{\rho_j+w_q}\rangle e_{\rho_j+w_q}\otimes e_{\rho_j}.\]  Then $E=\sum_{q=1}^\infty E_q$ and by Proposition \ref{snSum}, $||E|^s|_\Phi\leq 2^{1-s}V_2\sum_{q=1}^\infty ||E_q|^s|_\Phi$ for some $V_2>0$ that only depends on $\Phi$.  So using \eqref{calcE} and an argument almost identical to that used to prove \eqref{HnormBound} shows \[\begin{split}||E|^s|_\Phi&\leq 2^{1-s}V_2e^{\frac{-sR^2}{16}}\Phi\left(\left\{\left(\widetilde{\nu}_{\frac{1}{8}}(\rho_j)\right)^s\right\}_{j=1}^\infty\right)\left(\sum_{q=1}^\infty e^{\frac{-s|w_q|_\infty^2}{32}}\right).\end{split}\]  Since $\{w_q\}_{q=1}^\infty\subseteq \{a_j\}_{j=1}^\infty$, $\sum_{q=1}^\infty e^{\frac{-s|w_q|_\infty^2}{32}}\leq \sum_{j=1}^\infty e^{\frac{-s|a_j|_\infty^2}{32}}$.  Then by an argument similar to one given in the proof of Theorem \ref{Berezin}, $\sum_{j=1}^\infty e^{\frac{-s|a_j|_\infty^2}{32}}\leq\sum_{k=0}^\infty (2k+1)^{2n} e^{\frac{-s(rk)^2}{32}}$.  Hence \begin{equation}\label{TDE2}||E|^s|_\Phi\leq C_2e^{\frac{-sR^2}{16}}\Phi\left(\left\{\left(\widetilde{\nu}_{\frac{1}{8}}(\rho_j)\right)^s\right\}_{j=1}^\infty\right)\end{equation} where $C_2=2^{1-s}V_2\sum_{k=0}^\infty (2k+1)^{2n} e^{\frac{-s(rk)^2}{32}}$.

So by \eqref{TDE}, \eqref{TDE1} and \eqref{TDE2} \[e^{-sr^2(2n)}\Phi\left(\left\{\left(\widehat{\nu}_r\left(\rho_j\right)\right)^s\right\}_{j=1}^\infty\right)\leq WC_1|\left(T_\nu\right)^s|_\Phi+2^{1-s}WC_2e^{\frac{-sR^2}{16}}\Phi\left(\left\{\left(\widetilde{\nu}_{\frac{1}{8}}(\rho_j)\right)^s\right\}_{j=1}^\infty\right).\] 
It is also true by Proposition \ref{snSub} and Theorem \ref{Berezin} that \[\Phi\left(\left\{\left(\widetilde{\nu}_{\frac{1}{8}}(\rho_j)\right)^s\right\}_{j=1}^\infty\right)\leq \Phi\left(\left\{\left(\widetilde{\nu}_{\frac{1}{8}}(a_j)\right)^s\right\}_{j=1}^\infty\right)\leq C_3\Phi\left(\left\{\left(\widehat{\nu}_r(a_j)\right)^s\right\}_{j=1}^\infty\right)\] where $C_3=\sum_{j=0}^\infty (2j+1)^{2n}e^{-sr^2\alpha\left(j-2n\right)^2}$.  It follows that  \[e^{-sr^2(2n)}\Phi\left(\left\{\left(\widehat{\nu}_r\left(\rho_j\right)\right)^s\right\}_{j=1}^\infty\right)\leq WC_1|\left(T_\nu\right)^s|_\Phi+C_4e^{\frac{-sR^2}{16}}\Phi\left(\left\{\left(\widehat{\nu}_r(a_j)\right)^s\right\}_{j=1}^\infty\right)\] where $C_4=2^{1-s}C_2C_3W$.  Since $W$, $C_1$ and $C_4$ are each independent of $m$ and $\{\rho_j\}_{j=1}^\infty$ we have \begin{equation}\label{subLattConcl}e^{-sr^2(2n)}\Phi\left(\left\{\left(\widehat{\nu}_r\left(w\right)\right)^s\right\}_{w\in\Gamma_{y_p}}\right)\leq WC_1|\left(T_\nu\right)^s|_\Phi+C_4e^{\frac{-sR^2}{16}}\Phi\left(\left\{\left(\widehat{\nu}_r(a_j)\right)^s\right\}_{j=1}^\infty\right)\end{equation} for each $y_p\in A_m$.

By \eqref{LatticePartition} and basic properties of symmetric norming functions, \[\Phi\left(\{\left(\widehat{\nu}_r(a_j)\right)^s\}_{j=1}^\infty\right)\leq \sum_{q=1}^{m^{2n}}\Phi\left( \{\left(\widehat{\nu}_r(w)\right)^s\}_{w\in\Gamma_{y_q}}\right).\]  Then \eqref{subLattConcl} give us \[\begin{split}e^{-sr^2(2n)}\Phi\left(\{\left(\widehat{\nu}_r(a_j)\right)^s\}_{j=1}^\infty\right)&\leq e^{-sr^2(2n)}\sum_{q=1}^{m^{2n}}\Phi\left( \{\left(\widehat{\nu}_r(w)\right)^s\}_{w\in\Gamma_{y_q}}\right)\\&\leq m^{2n}WC_1|\left(T_\nu\right)^s|_\Phi+m^{2n}C_4e^{\frac{-sR^2}{16}}\Phi\left(\left\{\left(\widehat{\nu}_r(a_j)\right)^s\right\}_{j=1}^\infty\right).\end{split}\]  Recall that $R=mr$ where $m\geq2$.  Then choosing $m$ large enough so that $C_4m^{2n}e^{\frac{-s(mr)^2}{16}}\leq \frac{e^{-sr^2(2n)}}{2}$, we get from the above \begin{equation}\label{equN4}e^{-sr^2(2n)}\Phi\left(\{\left(\widehat{\nu}_r(a_j)\right)^s\}_{j=1}^\infty\right)\leq m^{2n}WC_1|\left(T_\nu\right)^s|_\Phi+\frac{e^{-sr^2(2n)}}{2}\Phi\left(\left\{\left(\widehat{\nu}_r(a_j)
\right)^s\right\}_{j=1}^\infty\right).\end{equation}  Thus for any positive Borel measure $\nu$ such that $\Phi\left(\left\{\left(\widehat{\nu}_r(a_j)
\right)^s\right\}_{j=1}^\infty\right)<\infty$, \eqref{equN4} gives us \begin{equation}\label{conclu}\Phi\left(\{\left(\widehat\nu_r(a_j)\right)^s\}_{j=1}^\infty\right) \leq m^{2n}2e^{sr^2(2n)}WC_1|\left(T_\nu\right)^s|_\Phi.\end{equation}  However a straightforward approximation argument shows that \eqref{conclu} holds for any positive Borel measure $\nu$ satisfying $|\left(T_\nu\right)^s|_\Phi<\infty$.  Thus with $V_1=m^{2n}2e^{sr^2(2n)}WC_1$, \[\Phi\left(\{\left(\widehat\nu_r(a_j)\right)^s\}_{j=1}^\infty\right) \leq V_1|\left(T_\nu\right)^s|_\Phi.\]

\section{Important Corollaries}

By Theorem \ref{Berezin}, Subsection \ref{sufficiency} and Subsection \ref{necessary}, we actually have the following corollary:  \begin{cor}
Let $0<s\leq 1$, $\alpha>0$, $r>0$ and $\{a_j\}_{j=1}^\infty$ be the $r$-lattice of $\mathbb{C}^n$.  Let $\nu$ be a positive Borel measure on $\mathbb{C}^n$ and $T_\nu$ be the corresponding Toeplitz operator.  Then there exists positive constants $B_1=B_1(r,s,n), B_2=B_2(r,\alpha,s,n)$, $B_3=B_3(r,\alpha,s,n)$ and $B_4=B_4(r,s,n)$ so that for any symmetric norming function $\Phi$, \[\begin{split}|\left(T_\nu\right)^s|_{\Phi}\leq B_1\Phi\left(\{\left(\widehat{\nu}_r(a_j)\right)^s\}_{j=1}^\infty\right)\leq B_1B_2\Phi\left(\{\left(\widetilde{\nu}_\alpha(a_j)\right)^s\}_{j=1}^\infty\right) &\leq B_1B_2B_3\Phi\left(\{\left(\widehat{\nu}_r(a_j)\right)^s\}_{j=1}^\infty\right)\\&\leq B_1B_2B_3B_4|\left(T_\nu\right)^s|_{\Phi}.\end{split}\]
\end{cor}

\subsection*{Acknowledgment}
The author would like to thank his dissertation committee for their help in writing this article: Dr. Jingbo Xia, Dr. Lewis Coburn and Dr. Ching Chou.

\bibliographystyle{plain}

\end{document}